\documentclass[10pt]{amsart}
\usepackage{cmap,amsmath,amsthm,amssymb,amscd} 
\usepackage[utf8]{inputenc}
\usepackage[english]{babel}
\usepackage{placeins}
\usepackage[T2A]{fontenc}
\usepackage{xcolor}
\usepackage{enumerate}
\usepackage{verbatim}
\usepackage{amsbsy}
\usepackage{graphics}
\usepackage[caption = false]{subfig}
\usepackage[final]{graphicx}
\usepackage{grffile}
\usepackage{mathtools}
\usepackage{tikz-cd}
\usetikzlibrary{decorations.pathmorphing}

\newtheorem{theorem}{{Theorem}}[section]
\newtheorem{proposition}{{Proposition}}[section]
\newtheorem{corollary}[theorem]{Corollary}
\newtheorem{lemma}[theorem]{Lemma}

\newtheorem{example}[theorem]{Example}

\newtheorem{definition}[theorem]{Definition}

\newcommand{\Z}{\mathbb{Z}}

\newcommand{\R}{\mathbb{R}}

\renewcommand{\iff}{\Leftrightarrow}

\renewcommand{\phi}{\varphi}

\newcommand{\id}{\operatorname{id}}
\newcommand{\tensor}{\otimes}
\newcommand{\rank}{\operatorname{rank}}

\title{Non-abelian tensor product and circular orderability of groups}
\author{Maxim~Ivanov}
\address{Sobolev Institute of Mathematics, 630090 Novosibirsk, Russia}
\email{m.ivanov2@g.nsu.ru}
\thanks{The work was performed according to the Government research assignment for IM SB RAS, project FWNF-2022-0004}

\numberwithin{equation}{section}
\begin{document}
\begin{abstract}
For a group $ G $ we consider its tensor square $G \tensor G$ and exterior square $G \wedge G$. 
We prove that for a circularly orderable group $G$, under some assumptions on $H_1(G)$ and $H_2(G)$, its exterior square and tensor square are left-orderable. This yields an obstruction for a circularly orderable group $G$ to have torsion. We apply these results to study circular orderability of tabulated virtual knot groups.
\end{abstract}
\maketitle
\tableofcontents

\section{Introduction} 
A group $G$ is called left-orderable if there is an order relation on elements of $ G $ such that
for all $ h, g_1, g_2 \in G$
	$$ g_1 < g_2 \iff \ hg_1 < hg_2.$$ 
When G is countable, it is left-orderable if and only if it can be embedded into $\text{Homeo}_+(\R)$.
A generalisation of this notion is circular orderability which is a way to state that a group may be "arranged on a circle" and the position of elements with respect to each other is preserved by left multiplication.
Group $ G $ is called circularly orderable if and only if it admits a map $c:\ G^3 \to \Z$ which satisfies the following:
\begin{enumerate}
	\item $c(g) \in \{-1,0,1\}, \ \forall g \in G$;
	\item $c^{-1}(0) = \{(g_1, g_2, g_3) \in G^3 \ | \ g_i = g_j \text{ for some } i \neq j \}; $
	\item $ c(g_2, g_3, g_4) - c(g_1, g_3, g_4) + c(g_1, g_2, g_4) - c(g_1, g_2, g_3) = 0, \ \forall g_1, g_2, g_3, g_4 \in G;$ 
	\item $ c(g_1, g_2, g_3) = c(hg_1, hg_2, hg_3), \ \forall h, g_1, g_2, g_3 \in G.$ 
\end{enumerate}
When G is countable, it is circularly orderable if and only if it can be embedded into $\text{Homeo}_+(S^1)$.
Every left-orderable group is circularly orderable but not vice versa. For example, finite cyclic groups are circularly orderable, but not left-orderable. Properties (3) and (4) mean that the map $ c $  is a homogenous 2-cocycle. Although the definition of a circular order by a homogenous 2-cocycle is quite intuitve, it is often convenient to pass to inhomogenous cocycles. In the rest of this paper we will understand circular orders as inhomogenous 2-cocycles $f_c: \ G^2 \to \Z$, satisfying
\begin{enumerate}
 \item $f_c(g,h) \in \{0, 1\}$, $\forall g, h \in G$;
 \item $f_c(g, g^{-1}) = 1 $, $\forall g \in G \setminus \{\id\}$ .
\end{enumerate}
These cocycles are in 1-1 correspondence with homogenous cocycles $ c $ defined above (see \cite{CG}).
\FloatBarrier
As always one can use a 2-cocycle to define a central extension of a group G.
\begin{figure}[!ht]
\centering
\begin{tikzcd}[ampersand replacement=\&]
	1 \arrow[r] \& \Z \arrow[r, ""] \& {G_{f}} \arrow[r,""] \& G \arrow[r, ""] \& 1
\end{tikzcd}
\end{figure}
\FloatBarrier
Thus for each circular order $f$ there is an associated group extension $G_{f}$. Moreover 	$ G_f $ is orderable, with $(1, \id_G)$ being a positive cofinal element.
Further, when $G$ is a left-ordered group, with $z$ being positive cofinal central element, $G /\langle z \rangle $ is circularly orderable and the circular order is given by a normalized section (\cite{Zeleva}). 

Even though there are circularly orderable but not left-orderable groups, in some cases, circular orderability may imply left-orderability. The simplest case is when the second cohomology group $H^2(G, \Z)$ is trivial, so $G$ is a subgroup of $G_f$ and hence it is left-orderable. It is natural to investigate the relation between these two notions. A strong result in this direction is obtained by J.~Bell, A.~Clay, T.~Ghaswala in~\cite{BCG}: The group is left-orderable if and only if $G \times \Z_n$ is circularly orderable for every natural $n$.

Theory of virtual knots and links was introduced by L.~Kauffman in \cite{Kau99} as a generalization of classical knot theory. A family of polynomial invariants of virtual knots was studied by G. Amrendra, M. Ivanov, M. Prabhakar, A. Vesnin in~\cite{FPoly}. These polynomials calculated for knots in J. Greens's table may be found in~\cite{FpolyCalc}. For a definition of a group of a virtual knot see~\cite{SGK}. A theorem of J.~Howie and H.~Short~\cite{HS} states that groups of classical knots are left-orderable. However, it is not true for virtual knots. For example, the group of a knot 3.7 from J. Green’s table~\cite{Green} contains torsion and thus it is not left-orderable. Another invariant of virtual knot is a virtual knot quandle. Circular orderability of quandles was recently studied by I.~Ba and M.~Elhamdadi in~\cite{BE}.

In the present paper we consider orderability of non-abelian tensor square and e{}xterior square functors (see \cite{BL}). The main result is stated in theorems~\ref{ExtOrd},~\ref{TensorOrd}. Providing some restrictions on a circularly orderable group $G$, its tensor square and exterior square are left-orderable. This gives a tool to check whether a torsion is an obstruction for a group to be circularly orderable. Finally, we apply these results to the groups of virtual knots in J.~Green's table. 

The author would like to thank A.~Vesnin for helpful discussions.

\section{Group Homology, Universal Extensions and Orderability}
In the following section we present results related to group homology which will be used in the later sections. Most examples of groups we are interested in have free abelianizations, so the next proposition comes in handy.
\begin{proposition} Let $G$ be a group with finite second homology, such that $G_{ab}$ is free abelian. 
Then $G$ is circularly orderable iff it is left-orderable.
\end{proposition}
\begin{proof} 
By universal coefficients we have a short exact sequence
	$$ 0 \to \text{Ext}(G_{ab}, \Z) \to H^2(G) \to \text{Hom}(H_2(G), \Z) \to 0$$
Since $G_{ab}$ is free abelian, $\text{Ext}(G_{ab}, \Z) = 0$. The second homology group $H_2(G)$ is finite, thus $\text{Hom}(H_2(G), \Z) = 0$ and $H^2(G) = 0$. If there is a circular order $f$, it represents $0$ in $H^2(G)$ and the corresonding extension is split. Then $G$ is orderable as a subgroup of an orderable group.  
\end{proof}

\begin{lemma}\label{ext_ord} Let there be a morphism of extensions.
\FloatBarrier
\begin{figure}[!ht]
\centering
\begin{tikzcd}[ampersand replacement=\&]
1 \arrow[r]\& A \arrow[r, ""] \arrow{d} \& \tilde G \arrow[d, "p"] \arrow[r, ""] \& G \arrow[d, "="] \arrow[r] \& 1 \\
1 \arrow[r] \& B \arrow[r,  ""] \& \widehat{G} \arrow[r, ""] \& G \arrow[r] \& 1
\end{tikzcd}
\end{figure}
\FloatBarrier 
\noindent
If $A$ is left-orderable and $ \widehat{G} $  is circularly orderable, then $\widetilde{G}$ is circularly orderable. If additionally, $ \widehat{G} $ is left-orderable, then $\widetilde{G}$ is left-orderable.
\end{lemma}
\begin{proof} 
From the diagram we get a short exact sequence
\FloatBarrier
\begin{figure}[!ht]
\centering
\begin{tikzcd}[ampersand replacement=\&]
1 \arrow[r] \& \text{Ker}(p) \arrow[r, ""] \& \widetilde{G} \arrow[r] \& \text{Im}(p) \arrow[r] \& 1,
\end{tikzcd}
\end{figure}
\FloatBarrier
$\text{Ker}(p)$ is orderable as a subgroup of an orderable group A. $\text{Im}(p)$ is circularly orderable as a subgroup of $\widehat{G}$ correspondingly. If $\widehat{G}$ is orderable, then so is $\text{Im}(p)$ and this yields the result.
\end{proof}
Now we apply this lemma to explore the connection of the second homology group with circular orderability.
Recall that for any group $G$ such that $ G_{\text{ab}} $ is free abelian, one can construct a universal central extension, or a Schur covering group, $\widehat{E}$ which is a central extension
\FloatBarrier
	\begin{figure}[!ht]
	\centering
	\begin{tikzcd}[ampersand replacement=\&]
	E_u: 1 \arrow[r] \& \ H_2(G) \arrow[r, ""] \& G_{\text{u}} \arrow[r, ""] \& G \arrow[r] \& 1
	\end{tikzcd}
	\end{figure}
\FloatBarrier
For the construction and related details see \cite{Kuz}. Universal central extension has the following property: for any other central extension $ E $ of a group $G$ there is a morphism of extensions $\phi: \ E_u \to E$. This allows us to conclude the following.
\begin{theorem}\label{univ-ord} Let $G$ be a circularly orderable group, with $H_1(G)$ being free abelian and $H_2(G)$ torsion-free.
Then its universal extension is left-orderable.
\end{theorem}
\begin{proof} Consider an extension, corresponding to a circular order $f$. Then there is a morphism from the universal extension.
\FloatBarrier
	\begin{figure}[!ht]
	\centering
	\begin{tikzcd}[ampersand replacement=\&]
	1 \arrow[r] \& H_2(G) \arrow[r, ""] \arrow[d, ""]\& G_{u} \arrow[d,""] \arrow[r, ""] \& G \arrow[d, "="] \arrow[r] \& 1\\
	1 \arrow[r] \& \Z \arrow[r, ""] \& G_{f} \arrow[r, ""] \& G \arrow[r] \& 1
	\end{tikzcd}
	\end{figure}
	\FloatBarrier
Since torsion-free abelian groups are left-orderable, applying Lemma~\ref{ext_ord} we get the left-orderability of $G_{u}$.
\end{proof}
There is a special case when $ H_2(G) = \Z $. Then the universal central extension may correspond to some circlular order on the group $G$.
\begin{proposition} Let $G$ be a group with $H_2(G) = \Z$, and $H_1(G)$ being free abeilan. If $ G $ is not left-orderable then its universal central extension $G_u$ corresponds to some circular order on $G$.
\end{proposition}
\begin{proof} $G$ satisfies the condition of the Theorem~\ref{univ-ord}, so $G_u$ is left-orderable. The only thing we need to check is that $H_2(G)$ is generated by a positive cofinal element in $G$. Since $G$ is not left-orderable, the morphism $i:\ H_2(G) \to \Z$ cannot be a zero morphism. We have $H_2(G) = \Z$, thus the morphism is an embedding and, by 5-lemma, so is $p: G_u \to G_{f}$. The image of a generator $1 \in \Z$ is $i(1) = k$ and it is obviously cofinal in $G_{f}$. Hence it is cofinal in $G_u$. If $1$ is not a positive generator, we replace it by $ -1 $, which is also cofinal. 
\end{proof}
Since we are interested in studying circular orderability of the universal extension, it is useful to know its second homology group.
\begin{lemma}\label{spectral_seq}
Suppose $G$ is a group, with $H_1(G) = \Z^k	$ and $H_2(G)$ being finite.
Consider a universal central extension $G_{u}$ of G, then $H_2(G_{u})$ is finite.
\end{lemma}
\begin{proof}{}
In Lyndon-Hochshild-Serre spectral sequence we have 
$$E^{2,0}_2 = H_2(G, H_0(A)) = H_2(G)$$
$$E^{1,1}_2 = H_1(G, H_1(A)) = H_1(G) \tensor A$$
$$E^{0,2}_2 = H_0(G, H_2(A)) = 0$$
Since $A = H_2(G)$ is finite then all these groups are finite and so groups $E^{2,0}_{\infty}, E^{1,1}_{\infty}, E^{0,2}_{\infty}$ are also finite. The second homology of the extension is an extension of finite group by finite, so it is finite.
\end{proof}
\section{Tensor product and exterior product}
The non-abelian tensor product of two groups, acting on each other, was introduced by R. Brown and J.-L. Loday \cite{BL}. It arises in the description of a third relative homotopy group of a triad. A group-theoretical approach to the tensor product along with computations for finite groups of order up to 30 is presented in \cite{BJR}. Here we will only need a special case of a tensor product, which is a tensor square of a group, with the action given by a conjugation.
\begin{definition} For an arbitraty group $G$ its tensor square $G \tensor G$ is defined as a group generated by the symbols $g \tensor h$, for $g, h \in G$, subject to the relations
	$$ gh \tensor k = ({}^gh \tensor {}^gk)(g \tensor k) $$
	$$ g \tensor hk = (g \tensor h)({}^hg \tensor {}^hk) $$
	where ${}^gh:= ghg^{-1}$.
\end{definition}
We will also consider an exterior square functor which is closely related to a tensor square functor.
\begin{definition} For a group $G$ its exterior square $G \wedge G$ is defined as a quotient of $G \tensor G$ by a normal subgroup generated by elements $\{g \tensor g \ | \ g \in G\}$.
\end{definition}
\begin{proposition}[Proposition 7, \cite{BJR}] For any central extension
\FloatBarrier
	\begin{figure}[!ht]
	\centering
	\begin{tikzcd}[ampersand replacement=\&]
	1 \arrow[r]\& A \arrow[r] \& K \arrow[r,  "\pi"] \& G \arrow[r] \& 1
	\end{tikzcd}
	\end{figure}
\FloatBarrier
There is a homomorphism $\xi: G \tensor G \to K$ such that $\pi \circ \xi$ is the commutator map 
$$g\tensor h \mapsto [g, h].$$
The map $\xi$ is defined by sending $g \tensor h$ to $[k_1, k_2]$, where $\pi(k_1) = g$, $\pi(k_2) = h$, and it does not depend on the choice of $k_i$.
\end{proposition}
Since $g\tensor g$ is mapped to zero by the commutator map, we have the following corollary.
\begin{corollary} The map $\xi$ factors through a map $\xi': \ G \wedge G \to K$ 
$$\xi'(g \wedge h) = [k_1, k_2].$$
\end{corollary}
\FloatBarrier
Tensor square and exterior square functors are connected by the following diagram
\FloatBarrier
\begin{figure}[!ht]
\centering
\begin{tikzcd}[ampersand replacement=\&]
\& \& 1 \arrow{d} \& 1 \arrow{d}\\
\& \Gamma(G_{ab}) \arrow{r} \arrow{d} \& J(G) \arrow{r} \arrow{d} \& H_2(G) \arrow{r} \arrow{d} \& 1 \\
1 \arrow{r} \& \nabla(G) \arrow{d} \arrow{r} \& G \tensor G \arrow{r} \arrow{d} \& G \wedge G \arrow{r} \arrow{d}\& 1\\
\& 1 \arrow{r} \& G' \arrow{d} \arrow{r} \& G' \arrow{d} \arrow{r} \& 1\\
\& \&  1 \& 1 \&
\end{tikzcd}
\end{figure}
\FloatBarrier
\noindent
where $\nabla(G)$ is a central subgroup generated by elements $\{g \tensor g \ | \  g \in G\}$ and $\Gamma(G_{ab})$ is a Whitehead's universal quadratic functor. Now we use the diagram to study orderability of exterior and tensor squares of a circularly orderable group $G$.

\begin{theorem}\label{ExtOrd}
Let $ G $ be a group with a circularly orderable or left-orderable central extension, such that $H_2(G)$ is torsion-free. Then $G \wedge G$ is circularly orderable or left-orderable correspondingly.
\end{theorem}
\begin{proof}
Consider the following diagram
\FloatBarrier
\begin{figure}[!ht]
\centering
\begin{tikzcd}[ampersand replacement=\&]
 1 \arrow[r] \& H_2(G) \arrow[r] \& G \wedge G \arrow[r,""] \arrow[d,"\xi'"] \& \left[G,G\right] \arrow[r] \arrow[d,"="] \& 1\\
\ \& \ \& \left[K,K\right] \arrow[r,"\pi"] \& \left[G,G\right] \arrow[r]\& 1
\end{tikzcd}
\end{figure}
\FloatBarrier
and a short exact sequence
\FloatBarrier
	\begin{figure}[!ht]
	\centering
	\begin{tikzcd}[ampersand replacement=\&]
	1 \arrow[r] \&\text{Ker}\ \xi' \arrow[r] \& G \wedge G \arrow[r, "\xi'"] \& \text{Im} \ \xi' \arrow[r] \& 1
	\end{tikzcd}
	\end{figure}
\FloatBarrier
If $K$ is circularly orderable (left-orderable) then so is $\text{Im} \ \xi'$ as a subgroup. From the commutativity of the diagram  we have
$$\text{Ker}\ \xi' \subset \text{Ker} \ \pi \circ \xi' = \text{Ker}\  G \wedge G \to [G,G] = H_2(G).$$
If $H_2(G)$ has no torsion, it is left-orderable. Then $  \text{Ker}\ \xi' $ is left-orderable.
$G \wedge G$ is an extension of circularly orderable group (left-orderable group) by a left-orderable group and thus it is circularly orderable (left-orderable).
\end{proof}
\begin{theorem}\cite{BFM}\label{TensorBFM}
Suppose that $G_{ab}$ has no elements of order 2. Then 
$$\nabla(G) \cong \nabla(G_{ab}) \text{ and }   G \tensor G \cong \nabla(G) \times (G \wedge G).$$
\end{theorem}
\begin{theorem}\label{TensorOrd} Let $G$ be circularly orderable, $H_1(G)$ be finitely generated free abelian, $H_2(G)$ be torsion-free. Then $G \tensor G$ is left-orderable.
\end{theorem}

\begin{proof}
Since abelianization $ G^{ab} = \Z^n$, by Theorem~\ref{TensorBFM}, we have $\nabla(G) \cong \nabla(G^{ab})$ and $\nabla(G)$ is a subgroup in $G^{ab} \tensor G^{ab}$ which is free abelian.  $G \tensor G$ is an extension of $G \wedge G$ by $\Z$ and $G \wedge G$ is orderable by Theorem~\ref{ExtOrd}. Therefore $G \tensor G$ is orderable.
\end{proof}
\begin{lemma}\label{torsion_element} If for some $a \in G$ we have $a^n = e$ and $[a, g]$ = e, then $(a \tensor g)^n = e$ in $ G \tensor G $.

\end{lemma}
\begin{proof}
Since $(a \tensor g)$ is central in $G \tensor G$, by relations in the tensor product, we have
	$$ e = a^n \tensor g = aa^{n-1} \tensor g = ({}^a(a^{n-1}) \tensor {}^ag) (a \tensor g) = (a^{n-1} \tensor g)  (a \tensor g) = (a \tensor g)^n$$
\end{proof}
\begin{corollary} Let $G$ be a group such that $H_1(G)$ is finitely generated free abelian and $H_2(G)$ is torsion-free. If there is a torsion element $a \in G$ and $a \tensor a \neq e \in G \tensor G$, then 
$G$ is not circularly orderable.
\end{corollary}
\begin{proof}
If $G$ is circularly orderable, then, by Theorem~\ref{TensorOrd}, $G \tensor G$ is left-orderable. By Lemma~\ref{torsion_element}, $a \tensor a$ is a torsion element in $G \tensor G$, which is an obstruction for a group to be left-orderable.
\end{proof}
A \emph{crystallographic group} is a group which contains a
finitely generated maximal abelian torsion-free subgroup of finite index. A
\emph{Bieberbach group} is a crystallographic group which is itself torsion-free. For an introduction to Bieberbach groups see~\cite{Charlap}. 

\begin{theorem}\label{bieberbach} Let $G$ be a group defined by a central extension
	\FloatBarrier
	\begin{figure}[!ht]
	\centering
	\begin{tikzcd}[ampersand replacement=\&]
	1 \arrow[r] \&  A \arrow[r, ""] \& G \arrow[r, ""] \& H \arrow[r] \& 1
	\end{tikzcd}
	\end{figure}
\FloatBarrier
\noindent
where $A$ is finitely generated, $H$ is finite. Then $ G \tensor G $ either has torsion or it is a Bieberbach group of a dimension $k \leq 2 \cdot \rank \ (A \tensor_\Z G_{ab})$.
\end{theorem}
\begin{proof}
For a central extension there is an exact sequence
\FloatBarrier
\begin{figure}[!ht]
\centering
\begin{tikzcd}[ampersand replacement=\&]
 (A \tensor G) \times (G \tensor A) \arrow[r, ""] \& G \tensor G \arrow[r, ""] \& H \tensor H \arrow[r] \& 1
\end{tikzcd}
\end{figure}
\FloatBarrier
$H$ is finite, so $H \tensor H$ is finite. $A$ is central, so $A$ and $G$ act trivially on each other.
This implies $A \tensor G = A \tensor_\Z G_{ab}$, $G \tensor A = G_{ab} \tensor_\Z A $. These groups are finitely generated abelian as tensor products of such groups. Both $H \tensor H$ and $A \tensor G$ are Noetherian as finite group and finitely generated abelian group correspondingly. Then $G \tensor G$ is Noetherian and hence finitely generated, because Noetherian groups are closed with respect to extensions by Noetherian groups. Suppose $G \tensor G$ is torsion-free. Denote by $ D $ the image of $(A \tensor G) \times (G \tensor A)$. $D$ is a torsion-free finitely generated subgroup of finite index in $G \tensor G$. Let $ L $ be a maximal abelian subgroup of $G$ that contains $D$. By Lagrange's theorem, $D$ is a finite index subgroup in $L$. Thus $\rank(D) = \rank(L)$. We obtain the following exact sequence
\FloatBarrier
\begin{figure}[!ht]
\centering
\begin{tikzcd}[ampersand replacement=\&]
1 \arrow[r] \& L \arrow[r] \& G \arrow[r] \& Q \arrow[r] \& 1
\end{tikzcd}
\end{figure}
\FloatBarrier\noindent
where $Q$ is finite and $L$ is maximal finitely generated torsion-free abelian subgroup of $G$, so $G$ is a Bieberbach group of a dimension $k = \rank(L)$. Since $ k $ is also a rank of the image, it is not greater then $2 \cdot \rank(A \tensor_\Z G_{ab}).$
\end{proof}

\begin{corollary} \label{bib_col} Let $G$ be group satisfying Theorem~\ref{bieberbach} and $G_{ab} \simeq A \simeq \Z$. Then $G \tensor G$ is torsion-free if and only if $G$ is isomorphic to $ \Z $ and $H \tensor H$ is cyclic or $G \simeq \Z \oplus \Z$.
\end{corollary}
\begin{proof} Denote by $D$ the image of $(A \tensor G) \times (G \tensor A)$ in $G$. If $G \tensor G$ has no torsion, then $D$ is equal to $0$, $\Z$, $\Z \oplus \Z$ or a Klein bottle group. If $D = 0$, then $G \tensor G $ is isomorphic to $H \tensor H$ and is finite. This contradicts with the fact that it maps surjectively on $ G_{ab} \tensor_\Z G_{ab} \simeq \Z $. If $ D = \Z $, then $G\tensor G$ is torsion-free and has a cyclic subgroup of finite index, so it is cyclic and $H \tensor H$ is cyclic as an image of $G \tensor G$. Finaly, when $ D \simeq \Z \oplus \Z$, it is known that the only Bieberbach groups of a dimension 2 are $\Z \oplus \Z$ and a fundamental group of a Klein bottle.
\end{proof}

\section{Virtual knot groups}
Theory of virtual knots and links was introduced by Kauffman in \cite{Kau99} as a generalization
of classical knot theory. One can define a virtual knot as an equivalence class of diagrams, which are 4-regular planar graphs with additional information in its vertices, corresponding to type of a crossing. Additionaly to classical crossings there appears a new one, which is called virtual and is depicted by a circle. For an introduction and diagramatic approach see \cite{Kau12}.
\FloatBarrier
 \begin{figure}[ht]
\centering
\includegraphics[height=1.2cm]{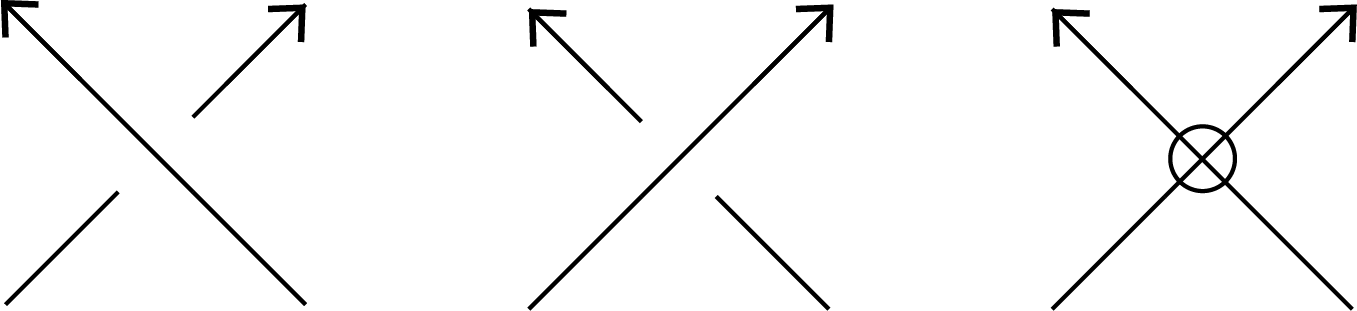}
\caption{Classical and virtual crossings.	}
\end{figure}
\FloatBarrier
When studying knots, one is usually interested in knot invariants. A group of a virtual knot is one of them. There are different group valued invariants of virtual knots, see \cite{VKG}.
Group $G_K$ of a virtual knot $K$ is defined by taking some diagram $D$ of a virtual knot $K$, assigning a generator to each arc in the diagram $D$ and imposing following relations for every crossing:
\FloatBarrier
 \begin{figure}[!ht]
\centering
\includegraphics[height=2.3cm]{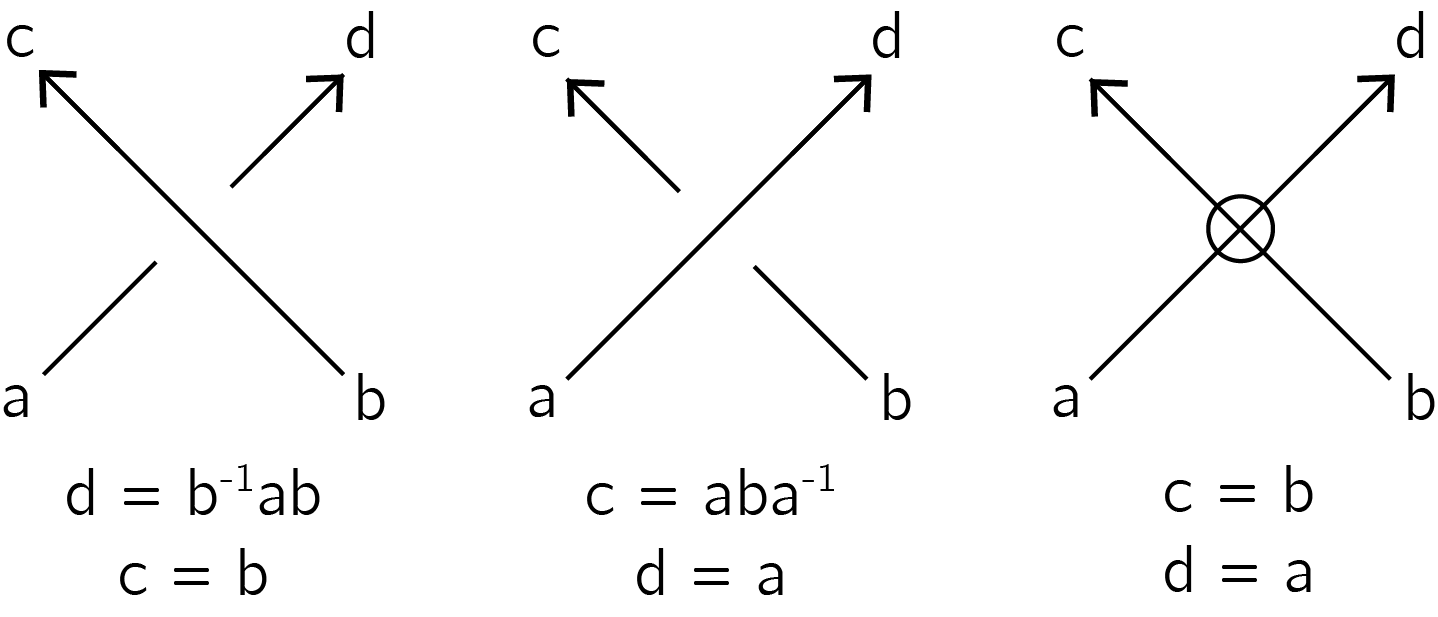}
\end{figure}
\FloatBarrier
The definition does not depend on a diagram and for classical knots the group $ G_K $ coincides with a fundamental group of a complement of a knot in $S^3$.
\begin{definition} Presentation $F/R = \langle x_1, \ ... \, \ x_n \ | \ r_1, ... r_m\rangle$  is called Wirtinger, if every $r_k$ has a form
	$$ x_i^{-1} \omega_k x_j \omega^{-1}_k $$
where $\omega_k$ is some word in $F$.
\end{definition}
Note that, by the construction, $(G_K)_{\text{ab}} = \Z$ and $G_K$ always has a Wirtinger presentation with the equal number of generators and relators.
\begin{definition} Deficiency of a group presentaion $\Gamma = \langle x_1, \ ... \, \ x_n \ | \ r_1, ... r_m\rangle$ is the number of generators minus the number of relators.
	$$ \text{def}(\Gamma) = n - m$$
\end{definition}
\begin{theorem}[Theorem 3, \cite{SGK}] Any Wirtinger presentation of deficiency 0 or 1 where all generators are conjugate can be realized as a virtual knot group.
\end{theorem}
A famous theorem of Howie and Short \cite{HS} states that groups of classical knots are left-orderable. However, it is not true for virtual knots. For example, the group of a knot 3.7 from J.~Green's table \cite{Green} contains torsion and hence is not left-orderable. Moreover there is a family of virtual knot groups, constructed by T.~Maeda.
\begin{example}[Maeda \cite{Maeda}] If G is a finitely presented group such that $G_{ab} = 
\Z$ and $G' = \Z_m$, then G is isomorphic to a group 
	$$ \langle x,\  a \ | \ a^m = 1, x^{-1}ax=a^n\rangle $$ for some integer $n$ with $(m,n) = (n-1, m) = 1$. Putting $y = xa^{-1}$ the group G becomes
$$G=\langle x, y \ | \ y = (y^{-1}x)^{-q}x(y^{-1}x)^q ,x = (y^{-1}x)^{-m}x(y^{-1}x)^{m}\rangle$$
where $q(n-1) + rm = 1$ ($q, r \in \Z$).
\end{example}
A universal extension of a virtual knot group is a Wirtinger group (Proposition 2, \cite{Kuz}), so universal extensions of virtual knot groups may give new examples of virtual knot groups. We have the following result.
\begin{proposition} Universal extension of a virtual knot group with nontrivial finite second homology is not circularly orderable.
\end{proposition}
\begin{proof} 
Universal extension $\overline{G}$ is a group with $H_1(\overline{G}) = \Z$ and, by Lemma~\ref{spectral_seq}, finite $H_2(\overline{G})$. Since it has $H_2(G)$ as a finite subroup, it is not left-orderable and hence it is not circularly orderable.
\end{proof}
In \cite{Ich} A. Ichimori computed groups of virtual knots in J.~Green's table and showed, that each group is isomorphic to either an infinite cyclic group or to some group $G_i$ from the following list:
$$
\begin{aligned}
& G_1=\langle x, y \mid x y x=y x y\rangle, \\ 
& G_2=\left\langle x, a \mid x^{-1} a x=a^{-1}, a^3 = 1	\right\rangle, \\
& G_3=\left\langle x, y \mid x y x=y x y, y=x^{-3} y x^3\right\rangle, \\
& G_4=\left\langle x, y \mid x^{-1} y x y^{-1} x=y x^{-1} y x y^{-1}\right\rangle, \\
& G_5=\left\langle x, a \mid x^{-1} a x=a^{-1}, a^5=1\right\rangle, \\
& G_6=\left\langle x, a \mid x^{-1} a x=a^2, a^5=1\right\rangle, \\
& G_7=\left\langle x, y \mid x^{-1} y x=y^{-1} x y\right\rangle, \\
& G_8=\left\langle x, a \mid x^{-1} a x=a^2, a^7=1\right\rangle, \\
& G_9=\left\langle x, y \mid x y x y^{-1} x^{-1}=y x^{-1} y x y^{-1}, x^{-1} y^{-1} x y x=y^{-1} x^{-1} y x y\right\rangle .
\end{aligned}
$$
$ G_1 $ is a trefoil knot group, $ G_4$ is a figure-eight knot group, hence they are left-orderable. So is $G_7$ as a group with one relation which is not torsion \cite{Brodskii}.
\begin{proposition} $G_3$ is not circularly orderable.
\end{proposition}
\begin{proof}
$$G_3=\left\langle x, y \mid x y x=y x y, y=x^{-3} y x^3\right\rangle$$
Consider a quotient by a central subgroup $\langle x^3 \rangle$
$$H = G_3 / \langle x^3 \rangle = \left\langle x, y \mid x=y^{-1}x^{-1}y x y, x^3, y^3\right\rangle.$$
Group $H$ is isomorphic to a group $\hat{A_4} = \left\langle a, b \mid a^3 = b^3 = (ab)^2\right\rangle$
and isomorphism is given by a map 
$$ x \mapsto a^2, y \mapsto b^4.$$
Since $G $ has a cyclic subgroup of finite index and $ G $ is not cyclic, there is a torsion in $ G $. If $H_2(G)$ is finite, then $ G $ is not circularly orderable. Suppose $H_2(G) \cong \Z $. By Theorem~\ref{TensorOrd}, $G \tensor G$ is left-orderable if $G$ is circularly orderable.
We have a natural epimorphism 	
$$ G \tensor G \mapsto H \tensor H. $$
$H \tensor H$ was computed in \cite{BJR} and it is isomorphic to $ \Z_3 \times Q $, where $ Q $ is the quaternion group. 
Using Corollary~\ref{bib_col} we obtain that the only possibility for $G \tensor G$ to be torsion-free is to be isomorphic to a fundamental group of a Klein bottle, because it cannot be abelian since its image $H \tensor H$ is not abelian. Finally we notice that $G \tensor G$ cannot be a Klein bottle group either as it is a central extension of a nilpotent group $H \tensor H$ and hence it should be nilpotent while the Klein bottle group is not. Thus in this case $ G \tensor G $ must have torsion and so $G$ is not circularly orderable.
\end{proof}
\begin{proposition} Let $G$ be a circularly orderable group with a presentation
	$$ G = \langle x, y \ | \ y^n = 1, xyx^{-1} = y^{l}\rangle $$
Then $G$ is abelian.
\end{proposition}
\begin{proof}
Suppose $G$ is left-orderable. A finite group $C$ generated by $ y $ is normal in G. Then, by Proposition~2.10 in \cite{CG}, the group $C$ is central in $G$, so $[x,y] = 1$. Then $G \cong G_\text{ab}$. 
\end{proof}
\begin{corollary}
	$ G_2, G_5, G_6, G_8 $ are not circularly orderable.
\end{corollary}
\begin{proposition} $G_9$ is left-orderable.
\end{proposition}
\begin{proof} In \cite{Ich} it was found that $G_9'$ has the following presentation
	$$ \langle a_k \  (k \in \Z) \ | \ a_{k+2} = a_k^{-2} a_{k+1}^{2}, \ [a_k, a_{k+1}] = 1 \rangle $$
Every finitely generated subgroup in $ G_9'$ is a subgroup of $H_{s,t}=\langle a_s, a_{s+1}, ..., a_t \rangle$ for some $s<t \in \Z$.  Using the relation $  a_{k+2} = a_k^{-2} a_{k+1}^{2} $, by induction, any element in $ H_{s,t}$ may be expressed as $a_{s-2}^na_{s-1}^m$ for some $n, m \in \Z$. Hence it is a subgroup in  $A_s = \langle a_{s-2}, a_{s-1} \rangle$. $A_s$ is torsion-free and thus left-orderable, because there is a homomorphism from $G_9'$ to a subgroup of complex numbers $ a_k \mapsto (1-i)^k $ since the following relation holds
	$$ 2((1-i)^{k+1} - (1-i)^k) = -2i(1-i)^k = (1-i)^2(1-i)^k = (1-i)^{k+2}.$$
and $a_{s-2}^na_{s-1}^m$ maps to $n(1-i)^{s-2} + m(1-i)^{s-1}$, which is equal to zero if and only if $n = m = 0.$ Since every finitely generated subgroup is left-orderable subgroup in $A_s$, by a Theorem~1.44 in \cite{CR} we obtain that $G_9'$ is left-orderable.
\end{proof}
As a result we have the following theorem.
\begin{theorem} If a group of a virtual knot in J. Green's table is orderable, then it is isomorphic to $\Z$, $G_1$, $G_4$, $G_7$ or $G_9$. Otherwise, it is not circularly orderable and it is isomorphic to $G_2$, $G_3$, $G_5$, $G_6$, or $G_8$.
\end{theorem}

\end{document}